\documentclass[a4paper,reqno]{amsart}
\usepackage[foot]{amsaddr}

\usepackage{amsmath}
\usepackage{amssymb}
\usepackage{amsthm}
\usepackage[colorlinks=false,citecolor=blue,breaklinks=true]{hyperref}

\usepackage[english]{isodate}
\usepackage{tikz}
\usetikzlibrary{arrows}
\tikzset{>=stealth}
\theoremstyle{plain}

\newtheorem{lemma}{Lemma}
\newtheorem{theorem}{Theorem}
\newtheorem{proposition}{Proposition}

\theoremstyle{definition}

\newcommand{\reals}{\mathbb{R}}

\renewcommand{\leq}{\leqslant}
\renewcommand{\geq}{\geqslant}
\newcommand{\vect}[1]{\mbox{\boldmath$#1$}}

\DeclareMathOperator{\BF}{BF}

\DeclareMathOperator{\rank}{rank}

\DeclareMathOperator{\supp}{supp}
\DeclareMathOperator{\mr}{mr}
\DeclareMathOperator{\pt}{pt}
\DeclareMathOperator{\ppt}{ppt}

\let\leq\leqslant
\let\geq\geqslant

\newcommand{\Vr}{V^{(r)}}
\newcommand{\Er}{E^{(r)}}
\newcommand{\Sr}{S^{(r)}}

\title{Minimum rank and zero forcing number for butterfly networks}

\author[D. Ferrero]{Daniela Ferrero$^1$}
\address{$^1$Department of Mathematics, Texas State University, San Marcos, U.S.A.}
\author[C. Grigorious]{Cyriac Grigorious$^{2,3}$}  
\author[T. Kalinowski]{Thomas Kalinowski$^{2,4}$}
\author[J. Ryan]{Joe Ryan$^2$}
\author[S. Stephen]{Sudeep Stephen$^{2,5}$}
\address{$^2$School of Mathematical and Physical Sciences, University of Newcastle, Callaghan, Australia}
\address{$^3$Graduate School, King’s College London, London, UK}
\address{$^4$School of Science and Technology, University of New England, Armidale, Australia}
\address{$^5$School of Mathematical Sciences, National Institute of Science Education and Research,
  Bhubaneswar, India}

\email{dferrero@txstate.edu}
\email{cyriac.grigorious@uon.edu.au}
\email{tkalinow@une.edu.au}
\email{joe.ryan@newcastle.edu.au}
\email{sudeep.stephen@niser.ac.in}

\date{}


\begin{document}

\begin{abstract}
  Zero forcing is a graph propagation process introduced in quantum physics and theoretical computer
  science, and closely related to the minimum rank problem. The minimum rank of a graph is the
  smallest possible rank over all matrices described by a given network. We use this relationship to
  determine the minimum rank and the zero forcing number of butterfly networks, concluding they
  present optimal properties in regards to both problems.
\end{abstract}

\maketitle

\section{Introduction}\label{sec:intro}

Let $G= (V,E)$ be a finite simple graph. Starting with a subset of the vertex set $V$ colored, we
consider the following coloring rule: an uncolored vertex is colored if it is the only uncolored
neighbor of some colored vertex. A vertex set $S\subseteq V$ is called \emph{zero-forcing} if,
starting with the vertices in $S$ colored and the vertices in the complement $V\setminus S$
uncolored, all the vertices can be colored by repeatedly applying the coloring rule. The minimum
cardinality of a zero-forcing set for the graph $G$ is called the \emph{zero-forcing number} of $G$,
denoted by $Z(G)$.

Zero forcing was introduced in linear algebra to study the problem of finding the minimum rank among
all symmetric matrices described by a graph~\cite{AIM-2008-Zeroforcingsets}. Rank minimization
problems consist of determining the minimum rank among all matrices whose off-diagonal zero-nonzero
pattern is determined by the edges of a graph. This problem is related to the inverse eigenvalue
problem~\cite{FallatHogben-2007-minimumranksymmetric} and with many problems in engineering
involving propagation of a signal through a network~\cite{Fazel2004}. Minimizing the rank of a
matrix is equivalent to maximizing its nullity, and the zero forcing number gives an upper bound for
the maximum nullity~\cite{AIM-2008-Zeroforcingsets}. For that reason, graphs in which the maximum
nullity coincides with their zero forcing number are of special interest~\cite{AIM-2008-Zeroforcingsets}.

Independently, the concept of zero forcing was also introduced in quantum physics, electrical
engineering and theoretical computer science. In physics, zero forcing was introduced to study
control of quantum systems and it is called graph infection~\cite{Burgarth2007,Severini2008}. The power domination problem in graph
theory~\cite{Haynes2002a} appeared in the study of the placement of monitoring units in electrical
power networks~\cite{Baldwin1993}, and it has been proven to be equivalent to the zero forcing
problem~\cite{Benson.etal_2015_Powerdominationzero}. Finally, the concept of zero forcing was
introduced as the fast-mixed search model for the study of fugitive search games on graphs. In
fugitive search games, a group of searchers, placed on the vertices of a graph, must find a fugitive
that is hiding in the vertices or edges of the graph~\cite{Dendris1997}. The different games are
determined by the allowed moves for the searchers and the fugitive. Depending on them, the minimum
number of required searchers reveals different graph properties~\cite{Bienstock1991}.

The fast-mixed model for graph search was introduced by Yang~\cite{Yang2013} as a
combination of the fast method~\cite{Dyera} and the mixed search method~\cite{Bienstock1991a}. The
fast-mixed number of a graph is the minimum number of searchers required to find a fugitive in the
graph and it coincides with the zero forcing number.

In this paper, we prove that butterfly interconnection networks have optimal minimum rank and zero
forcing properties, and as a consequence, optimal fast-search number. The interest on this
particular family of graphs is that they provide an excellent model for interconnection networks
where search problems are used to detect faulty nodes or false information~\cite{Dobrev2006,Kirousis2000}.

In order to make a more precise statement about the relation between the zero forcing number and the
minimum rank problem, let $S_n(F)$ denote the set of symmetric $n \times n$ matrices over a field
$F$. For a simple graph $G=(V,E)$ with vertex set $V=\{1,\dotsc,n\}$, let $S(F,G)$ be the set of
matrices in $S_n(F)$ whose non-zero off-diagonal entries correspond to edges of $G$, that is,
\[S(F,G)=\{A \in S_{n}(F)\ :\ i\neq j\implies \left(ij\in E(G)\iff a_{ij}\neq 0\right)\}.\]  
The minimum $F$-rank of $G$ is defined as the minimum rank over all matrices $A$ in $S(F,G)$:
\[\mr^F(G)=\min\left\{\rank(A)\ :\ A\in S(F,G)\right\}.\]
If the index $F$ is omitted then it is understood that $F=\reals$. The link between the zero forcing
number and the minimum rank problem is established by the observation that for a zero-forcing set
$S$ and a matrix $A\in S(F,G)$, the rows of $A$ that correspond to the vertices in $V\setminus S$
must be linearly independent, so $\rank(A)\geqslant n-\lvert S\rvert$, and consequently
\begin{equation}\label{eq:Zfnullity}
\mr^F(G)\geqslant n-Z(G).
\end{equation}
Based on this insight, the authors of~\cite{AIM-2008-Zeroforcingsets} determined $\mr(G)$ for
various graph classes and established equality in~\eqref{eq:Zfnullity}, independent of the field
$F$, in many cases. In~\cite{HuangChangYeh2010}, the same is proved for block-clique graphs and unit
interval graphs. Recently, the zero forcing number of cartesian products of cycles was established
by constructing a matrix in $S(F,G)$ with the required
rank~\cite{Benson.etal_2015_Powerdominationzero}. The American Institute for Mathematics maintains
the minimum rank graph catalog~\cite{HogbenBarrettGroutHolstRasmussenSmith.-2016-AIMminimumrank} in
order to collect known results about the minimum rank problem for various graph classes.

The rest of the paper is structured as follows. Section~\ref{sec:notation} contains some notation
and a precise statement of our main result. In Section~\ref{sec:upper_bound}, we prove an upper bound
for the zero-forcing number of the butterfly network by an explicit construction of the
corresponding zero forcing set $S$. By~\eqref{eq:Zfnullity}, this implies a lower bound for the
minimum rank of the butterfly network, and in Section~\ref{sec:upper_bound} we establish that this
bound is tight by showing that the rows of the adjacency matrix corresponding to the vertices in
the complement of the zero-forcing set span the row space of the adjacency matrix of the butterfly
network (over any field $F$).

\section{Notation and main result}\label{sec:notation}

Let $G= (V,E)$ be a finite simple graph. For a vertex $v \in V$, the open neighborhood of $v$ is the
set $N(v) = \{u : uv \in E(G)\}$ and the closed neighborhood of $v$ is the set
$N[v] = N(v) \cup \{v\}$. We denote by $I_n$ the $n \times n$ identity matrix, and we use $I$ for
$I_n$ when the order $n$ is clear from the context.

For a positive integer $r$, the butterfly network $\BF(r)=\left(V^{(r)},\,E^{(r)}\right)$ has vertex
set $\Vr =\Vr_0\cup\Vr_1\cup\cdots\cup\Vr_r$ and edge set $\Er =\Er_1\cup\Er_2\cup\cdots\cup\Er_r$,
where
\begin{align*}
\Vr_i&=\left\{(\vect x,i)\ :\ \vect x\in\{0,1\}^r\right\}&&\text{for }i=0,1,\dotsc,r,\\
\Er_i&=\left\{\left\{(\vect x,i-1),\,(\vect y,i)\right\}\ :\ \vect x\in\{0,1\}^r,\,\vect y\in\{\vect
  x,\vect x+\vect e_i\}\right\}&&\text{for }i=1,2,\dotsc,r.
\end{align*}
Here addition is modulo 2, and $\vect e_i$ is the binary vector of length $r$ with a one in position
$i$ and zeros in all other components. For convenience, we identify the binary vector
$\vect x=(x_1,\dotsc,x_r)\in\{0,1\}^r$ with the number $x_12^0+x_22^1+\dots+x_r2^{r-1}$. Using this
identification the butterfly network $\BF(4)$ is shown in Figure~\ref{fig:butterfly}.
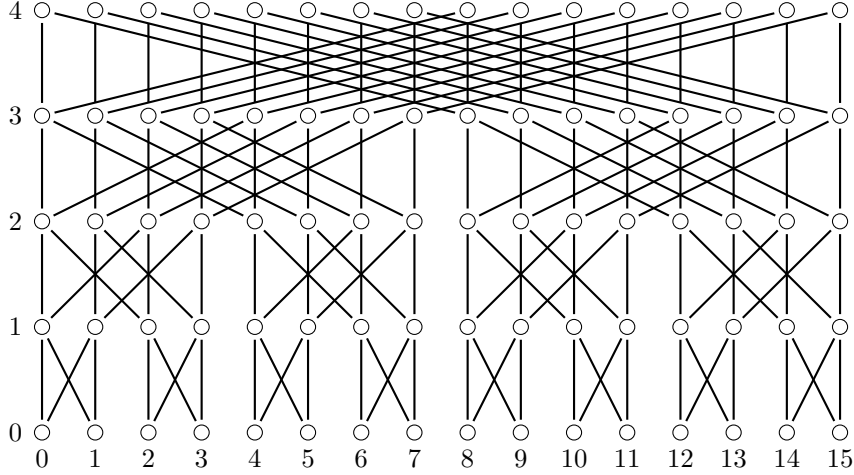
\begin{figure}[htb]
  \centering
  \begin{tikzpicture}[scale=.7]
  \foreach \x in {0,...,15}
    \foreach \y in {0,...,4} 
       {\pgfmathtruncatemacro{\label}{\x + 16 * \y}
       \node [circle,draw,fill=none,inner sep=2pt,outer sep = 2pt]  (\label) at (\x,2*\y) {};} 

  \foreach \x in {0,...,15}
    \foreach \y in {0,...,3}  
      {\pgfmathtruncatemacro{\tail}{\x + 16 * \y}
      \pgfmathtruncatemacro{\head}{\tail + 16}
      \draw[thick] (\tail)--(\head);}    
    
  \foreach \x in {0,...,7}  
     \foreach \k in {0,1}  
       {\pgfmathtruncatemacro{\tail}{2 * \x + \k} 
       \pgfmathtruncatemacro{\head}{\tail + 17 - 2 * \k}
       \draw[thick] (\tail)--(\head);
       \pgfmathtruncatemacro{\tail}{48 + \x + 8 * \k} 
       \pgfmathtruncatemacro{\head}{\tail + 24 - 16 * \k}
       \draw[thick] (\tail)--(\head);
       }
  
  \foreach \x in {0,...,3}  
     \foreach \k in {0,1}
       \foreach \l in {0,1}
         {\pgfmathtruncatemacro{\tail}{16 + 4 * \x + 2*\k + \l} 
         \pgfmathtruncatemacro{\head}{\tail + 18 - 4 * \k}
         \draw[thick] (\tail)--(\head);}
 
  \foreach \x in {0,1}  
     \foreach \k in {0,1}
       \foreach \l in {0,...,3}
         {\pgfmathtruncatemacro{\tail}{32 + 8 * \x + 4*\k + \l} 
         \pgfmathtruncatemacro{\head}{\tail + 20 - 8 * \k}
         \draw[thick] (\tail)--(\head);}
  
  \foreach \x in {0,...,15} \node [draw=none] (label\x) at (\x,-.5) {\x};
  \foreach \y in {0,...,4} \node [draw=none] (label\y) at (-.5,2*\y) {\y};
\end{tikzpicture}
  \caption{The butterfly network $\BF(4)$. Note the recursive structure: the bottom two rows consist of eight
      copies of $\BF(1)$, the bottom three rows of four copies of $\BF(2)$, and the bottom four rows of two copies of $\BF(3)$.}\label{fig:butterfly}
\end{figure}
As an example, the edges $\{(14,1),\,(14,2)\}$ and $\{(14,1),\,(12,2)\}$ in
Figure~\ref{fig:butterfly} correspond to 
$\{((0,1,1,1),1),\,((0,1,1,1),2)\}\in E_2^{(4)}$ and
$\{((0,1,1,1),1),\,((0,0,1,1),2)\}\in E_2^{(4)}$. Note, that the order and the size of $\BF(r)$ are
given by
\begin{align*}
  \left\lvert\Vr\right\rvert&=\sum_{i=0}^r\left\lvert\Vr_i\right\rvert=(r+1)2^r,&
  \left\lvert\Er\right\rvert&=\sum_{i=1}^r\left\lvert\Er_i\right\rvert=r2^{r+1}.
\end{align*}

Our main result is the following theorem.
\begin{theorem}\label{thm:main_result}
The minimum rank of the butterfly network $\BF(r)$ over any field $F$ equals
\[\mr^{F}\left(\BF(r)\right) =\frac29\left[(3r+1)2^{r}-(-1)^r\right],\]
and this is equal to the rank of the adjacency matrix of $\BF(r)$. Furthermore, for the butterfly
network we have equality in~\eqref{eq:Zfnullity}, i.e.,
\[Z\left(\BF(r)\right)=(r+1)2^r-\mr^F\left(\BF(r)\right)=\frac19\left[(3r+7)2^r+2(-1)^r\right].\] 
\end{theorem}

\section{The upper bound for $Z(\BF(r))$}\label{sec:upper_bound}
Let $(J_n)$ denote the Jacobsthal sequence\footnote{OEIS:\href{http://oeis.org/A001045}{A001045}}
which is defined by $J_0=0$, $J_1=1$ and $J_n=J_{n-1}+2J_{n-2}$ for $n\geqslant 2$. We will need the
relation
\begin{equation}\label{eq:Jacobsthal_bound}
  J_{n+2}=2^n+J_n\quad\text{for every integer }n\geqslant 0,  
\end{equation}
which can be seen as follows. We first use induction on $n$ to verify $J_n+J_{n+1}=2^n$ for all
$n\geq 0$. The base case is $J_0+J_1=1=2^0$, and for $n\geq 1$,
\[J_n+J_{n+1}=J_n+(J_n+2J_{n-1})=2(J_{n-1}+J_{n})=2\cdot 2^{n-1}=2^n.\]
Then~\eqref{eq:Jacobsthal_bound} follows from the recursive definition:
$J_{n+2}=J_{n+1}+2J_n=(J_{n}+J_{n+1})+J_n=2^n+J_n$.

For every $r$, we define a set $\Sr=\Sr_0\cup\Sr_1\cup\cdots\cup\Sr_r$ by
\[\Sr_i = \left\{(x,i)\in\Vr_i\ :\ 2^{i+1}\ell\leqslant x\leqslant
        2^{i+1}\ell+J_{i+1}-1\text{ for some }\ell\in\{0,1,\dots,\lfloor (2^r-1)/2^{i+1}\rfloor\}\right\}  \]
    for $i=0,1,\dotsc,r$. The range for $\ell$ comes from the observation that
    \[\lfloor (2^r-1)/2^{i+1}\rfloor=\max\{\ell\,:\,(2^{i+1}\ell,i)\in\Vr\}.\]
    In the following we will use the phrase ``for some $\ell$'' as a shorthand for ``for some
    non-negative integer $\ell$'', and the largest relevant value of $\ell$ will always be implicit
    in the requirement that a certain $(x,i)$ is a vertex of $\BF(r)$. For $r=4$ the construction of
    the set $\Sr$ is illustrated in Figure~\ref{fig:propagating}.
\begin{figure}[htb]
  \centering
  \begin{tikzpicture}[scale=.7]
  \foreach \x in {0,...,15}
    \foreach \y in {0,...,4} 
       {\pgfmathtruncatemacro{\label}{\x + 16 * \y}
       \node [circle,draw,fill=none,inner sep=2pt,outer sep = 2pt]  (\label) at (\x,2*\y) {};} 

  \foreach \x in {0,...,15}
    \foreach \y in {0,...,3}  
      {\pgfmathtruncatemacro{\tail}{\x + 16 * \y}
      \pgfmathtruncatemacro{\head}{\tail + 16}
      \draw[thick] (\tail)--(\head);}    
    
  \foreach \x in {0,...,7}  
     \foreach \k in {0,1}  
       {\pgfmathtruncatemacro{\tail}{2 * \x + \k} 
       \pgfmathtruncatemacro{\head}{\tail + 17 - 2 * \k}
       \draw[thick] (\tail)--(\head);
       \pgfmathtruncatemacro{\tail}{48 + \x + 8 * \k} 
       \pgfmathtruncatemacro{\head}{\tail + 24 - 16 * \k}
       \draw[thick] (\tail)--(\head);
       }
  
  \foreach \x in {0,...,3}  
     \foreach \k in {0,1}
       \foreach \l in {0,1}
         {\pgfmathtruncatemacro{\tail}{16 + 4 * \x + 2*\k + \l} 
         \pgfmathtruncatemacro{\head}{\tail + 18 - 4 * \k}
         \draw[thick] (\tail)--(\head);}
 
  \foreach \x in {0,1}  
     \foreach \k in {0,1}
       \foreach \l in {0,...,3}
         {\pgfmathtruncatemacro{\tail}{32 + 8 * \x + 4*\k + \l} 
         \pgfmathtruncatemacro{\head}{\tail + 20 - 8 * \k}
         \draw[thick] (\tail)--(\head);}
  
  \foreach \x in {0,...,15} \node [draw=none] (label\x) at (\x,-.5) {\x};
  \foreach \y in {0,...,4} \node [draw=none] (label\y) at (-.5,2*\y) {\y};

  \foreach \x in {0,...,7}
     \node [circle,draw,fill=black,minimum size=8pt]  () at (2*\x,0) {};
  \foreach \x in {0,...,3}
     \node [circle,draw,fill=black,minimum size=8pt]  () at (4*\x,2) {};
  \foreach \y in {0,...,2}
     {\node [circle,draw,fill=black,minimum size=8pt]  () at (\y,4) {};
      \node [circle,draw,fill=black,minimum size=8pt]  () at (8+\y,4) {};}
  \foreach \y in {0,...,4}
     \node [circle,draw,fill=black,minimum size=8pt]  () at (\y,6) {};
  \foreach \y in {0,...,10}
     \node [circle,draw,fill=black,minimum size=8pt]  () at (\y,8) {};
\end{tikzpicture}
  \caption{The set $S^{(4)}$ indicated by filled vertices.}\label{fig:propagating}
\end{figure}
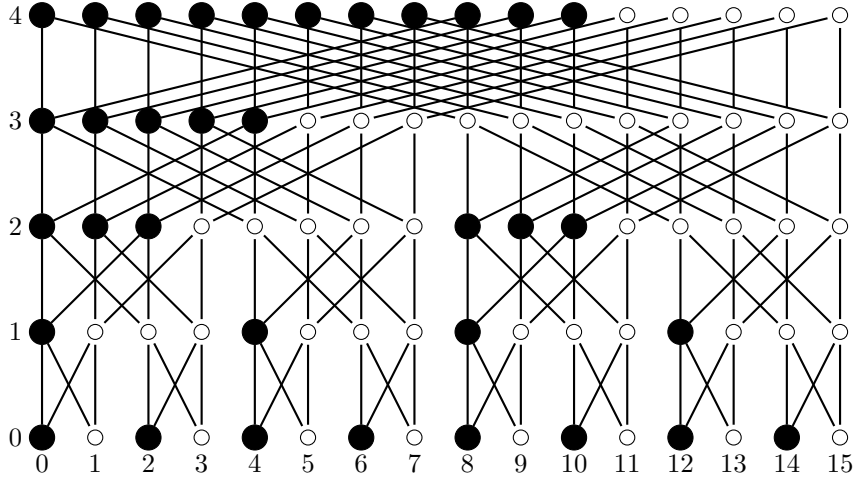

Solving the recurrence relation for the numbers $J_n$, we find $J_n=(2^n-(-1)^n)/3$, and this
implies the following closed form expression for the size of the set $\Sr$.
\begin{lemma}\label{lem:cardinality}
  We have $\displaystyle\left\lvert\Sr\right\rvert=J_{r+1}+\sum_{i=1}^r2^{r-i}J_i=\frac19\left[(3r+7)2^r+2(-1)^r\right]$.\qed
\end{lemma}
Next we want to verify that $\Sr$ is a zero forcing set for $\BF(r)$. For this purpose we set
$X_0=\Sr$ and define a sequence $X_1,X_2,\dots,X_{2r}$ of vertex sets by
\begin{align}
  X_k &= X_{k-1}\cup\left\{(x,r-k)\in\Vr_{r-k}\ :\ \{(x,r-k)\}=N(v)\setminus X_{k-1}\text{ for some
    }v=(y,r-k+1)\in X_{k-1}\right\} \label{eq:prop_down}\\
  X_{r+k} &= X_{r+k-1}\cup\left\{(x,k)\in\Vr_{k}\ :\ \{(x,k)\}=N(v)\setminus X_{r+k-1}\text{ for some
    }v=(y,k-1)\in X_{r+k-1}\right\} \label{eq:prop_up}
\end{align}
for $k=1,\dots,r$. These sets can be interpreted in terms of the forcing process as
  follows. The sets $X_0$ to $X_r$ correspond to forcing downwards: $X_0$ is the set of initially colored vertices, $X_1$ is the set of colored vertices
  after the level $r$ vertices in $X_0$ have been used to force vertices on level $r-1$, $X_2$ is
  the set of colored vertices after the level $r-1$ vertices in $X_1$ have been used to force
  vertices on level $r-2$, etc., up to $X_r$, which is the set of colored vertices after the level
  $1$ vertices in $X_{r-1}$ have been used to force vertices on level $0$. Then we turn around and
  force upwards:
  $X_{r+1}$ is the set of colored vertices after the level $0$ vertices in $X_{r}$ have been used
  to force vertices on level $1$, $X_{r+2}$ is the set of colored vertices after the level $1$ vertices in $X_{r+1}$ have been used
  to force vertices on level $2$, etc., up to $X_{2r}$, which is the set of colored vertices after the level
  $r-1$ vertices in $X_{2r-1}$ have been used to force vertices on level $r$. As a consequence, in
  order to prove that $X_0=\Sr$ is a zero forcing set, it is sufficient to prove that
  $X_{2r}=\Vr$. This is the purpose of the following four lemmas. Lemma~\ref{lem:prop_down_aux}
  describes the downward forcing step from level $r-k+1$ to level $r-k$, which is then used inductively in
  Lemma~\ref{lem:prop_down} to obtain an explicit description of the sets $X_0$ to $X_r$. Similarly, Lemma~\ref{lem:prop_up_aux}
  describes the upward forcing step from level $k-1$ to level $k$, which is then used inductively in
  Lemma~\ref{lem:prop_up} to obtain an explicit description of the sets $X_r$ to $X_{2r}$.
\begin{lemma}\label{lem:prop_down_aux}
  Fix $i\in\{0,1,\dots,r-1\}$, and let $Y=Y(i)\cup Y(i+1)\cup\dots\cup Y(r)\subseteq\Vr$ be defined by
  \begin{align*}
Y(i) &= \Sr_i=\left\{(x,i)\in\Vr_i\ :\ 2^{i+1}\ell\leqslant x\leqslant
        2^{i+1}\ell+J_{i+1}-1\text{ for some }\ell\right\} \\
Y(j) &= \left\{(x,j)\in\Vr_j\ :\ 2^{j}\ell\leqslant x\leqslant 2^{j}\ell+J_{j+1}-1\text{ for some }\ell\right\} &&\text{for }j\in\{i+1,\dotsc,r\}.
  \end{align*}
  Furthermore, let $Z$ be the set of vertices on level $i$ that can be forced by vertices in
  $Y(i+1)$ when $Y$ is the set of colored vertices,
  that is, $Z = \left\{(x,i)\ :\ \{(x,i)\}=N(v)\setminus Y\text{ for some
      }v=(y,i+1)\in Y\right\}$. Then
  \begin{equation}\label{eq:target}
    Y(i)\cup Z=\left\{(x,r-k)\in\Vr_{i}\ :\ 2^{i}\ell\leqslant x\leqslant 2^{i}\ell+J_{i+1}-1\text{ for some }\ell\right\}.
  \end{equation}
\end{lemma}
\begin{proof}
  Let $(x,i)$ be an arbitrary element of the right hand side of~(\ref{eq:target}), say
  $2^i\ell\leq x\leq 2^{i}\ell+J_{i+1}-1$. We need to show that $(x,i)\in Y(i)\cup Z$. If $\ell$ is
  even, then $2^{i+1}(\ell/2)\leqslant x\leqslant 2^{i+1}(\ell/2)+J_{i+1}-1$ and
  $(x,i)\in \Sr_i=Y(i)$. Otherwise, $\ell=2\ell'+1$ for some integer $\ell'$ and
  $2^{i+1}\ell'+2^i\leqslant x\leqslant 2^{i+1}\ell'+2^i+J_{i+1}-1$. For $y=x-2^i$, the vertex
  $(y,i+1)$ is in $Y$ because
\[2^{i+1}\ell'\leqslant y\leqslant 2^{i+1}\ell'+J_{i+1}-1\leqslant 2^{i+1}\ell'+J_{i+2}-1.\]
We want to show that $(y,i+1)$ forces $(x,i)$. Let $\ell''=\lfloor \ell'/2\rfloor$, so that $\ell'=2\ell''+\varepsilon$ with $\varepsilon\in\{0,1\}$. The neighborhood of $(y,i+1)$ is
\[N((y,i+1))=
\begin{cases}
\{(y,i),(x,i)\} & \text{if }i=r-1,\\
\{(y,i),(x,i),(y,i+2),(y+(-1)^\varepsilon2^{i+1},i+2)\} & \text{if }i\leq r-2. 
\end{cases}
\]
Now $(y,i)\in Y$, and for $i=r-1$ that's all we need. Using~\ref{eq:Jacobsthal_bound}, we have
\[2^{i+2}\ell''=2^{i+1}(\ell'-\varepsilon)\leqslant 2^{i+1}\ell'\leqslant y
\leqslant 2^{i+1}\ell'+J_{i+1}-1=2^{i+2}\ell''+2^{i+1}\varepsilon+J_{i+1}-1\leqslant
2^{i+2}\ell''+J_{i+3}-1,\]
and therefore $(y,i+2)\in Y$. Similarly, 
\[2^{i+2}\ell''=2^{i+1}\ell'+(-1)^\varepsilon2^{i+1}\leqslant y+(-1)^\varepsilon2^{i+1}\leqslant
2^{i+2}\ell''+2^{i+1}+J_{i+1}-1\leqslant 2^{i+2}\ell''+J_{i+3}-1,\]
and therefore $(y+(-1)^\varepsilon2^{i-1},i+2)\in Y$. Consequently,
$\{(x,i)\}=N((y,i+1))\setminus Y$, and this implies $(x,i)\in Z$. So, we have verified
\[Y(i)\cup Z \supseteq \left\{(x,i)\in\Vr_i\ :\ 2^{i}\ell\leqslant x\leqslant 2^{i}\ell+J_{i+1}-1\text{ for some
    }\ell\right\}.\]
To prove the other inclusion, consider any vertex $(x,i)\in\Vr_i$ which is not contained in the right hand
side of~(\ref{eq:target}), that is, $2^i\ell+J_{i+1}\leqslant x\leqslant 2^i(\ell+1)-1$ for some
$\ell=2\ell'+\varepsilon$. In particular, $(x,i)\not\in Y(i)$, and it remains to be checked that 
$(x,i)\not\in Z$, that is, $(x,i)$ cannot be forced. We have
\[N((x,i))\cap\Vr_{i+1}=\{(x,i+1),(x+(-1)^\varepsilon 2^i,i+1)\}.\]
If $(x,i+1)\in Y$ then
\begin{align*}
  \ell\equiv 0\text{ or }1\pmod 4 &\implies (x+2^{i+1},i+2)\in N((x,i+1))\setminus Y,\\
  \ell\equiv 2\text{ or }3\pmod 4 &\implies (x,i+2)\in N((x,i+1))\setminus Y
\end{align*}
If $(x+(-1)^\varepsilon 2^i,i+1)\in Y$ then $\ell$ is odd and
\begin{align*}
  \ell\equiv 1\pmod 4 &\implies (x+2^{i+1},i+2)\in N((x,i+1))\setminus Y,\\
  \ell\equiv 3\pmod 4 &\implies (x,i+2)\in N((x,i+1))\setminus Y.
\end{align*}
In all cases it follows that $(x,i)\not\in Z$, and this concludes the proof.
\end{proof}
\begin{lemma}\label{lem:prop_down}
For $k\in\{0,1,\dotsc,r\}$, $X_k=X_k(0)\cup X_k(1)\cup\cdots\cup X_k(r)$ with
\begin{align*}
X_k(i) &= \Sr_i &&\text{for }i\in\{0,1,\dotsc,r-k\}\\
X_k(i) &= \left\{(x,i)\in\Vr_i\ :\ 2^{i}\ell\leqslant x\leqslant 2^{i}\ell+J_{i+1}-1\text{ for some }\ell\right\} &&\text{for }i\in\{r-k+1,\dotsc,r\}.
\end{align*}
In particular, $X_r(0)=\Vr_0$.
\end{lemma}
\begin{proof}
We proceed by induction on $k$. For $k=0$, there is nothing to do since
$X_0=\Sr=\Sr_0\cup\cdots\cup\Sr_r$. Let $k\geqslant 1$ and set $i=r-k$. By~(\ref{eq:prop_down}), the sets
$X_k$ and $X_{k-1}$ differ only on level $i$, that is, $X_k(j)=X_{k-1}(j)$ for all $j\neq i$. By induction, this implies
\begin{align*}
X_k(j) &= \Sr_j &&\text{for }j\in\{0,1,\dotsc,i-1\}\\
X_k(j) &= \left\{(x,j)\in\Vr_j\ :\ 2^{j}\ell\leqslant x\leqslant 2^{j}\ell+J_{j+1}-1\text{ for some }\ell\right\} &&\text{for }j\in\{i+1,\dotsc,r\}.
\end{align*}
Also by induction,
\begin{align*}
X_{k-1}(j) &= \Sr_j &&\text{for }j=i\\
X_{k-1}(j) &= \left\{(x,j)\in\Vr_j\ :\ 2^{j}\ell\leqslant x\leqslant 2^{j}\ell+J_{j+1}-1\text{ for some }\ell\right\} &&\text{for }y\in\{i+1,\dotsc,r\},
\end{align*}  
and we can apply Lemma~\ref{lem:prop_down_aux} with $Y=X_{k-1}(i)\cup X_{k-1}(i+1)\cup\dots\cup
X_{k-1}(r)$ to obtain
\begin{multline*}
  X_k(i) \stackrel{\eqref{eq:prop_down}}{=} X_{k-1}(i)\cup\left\{(x,i)\in\Vr_{i}\ :\ \{(x,i)\}=N(v)\setminus X_{k-1}\text{ for some
    }v=(y,i+1)\in X_{k-1}\right\}\\
  =\left\{(x,i)\in\Vr_i\ :\ 2^{i}\ell\leqslant x\leqslant 2^{i}\ell+J_{i+1}-1\text{ for some }\ell\right\}.\qedhere
\end{multline*}
\end{proof}
\begin{lemma}\label{lem:prop_up_aux}
  Fix $i\in\{1,2,\dots,r\}$, and let $Y=Y(0)\cup Y(1)\cup\dots\cup Y(i)$ be defined by
  \begin{align*}
    Y(j) &= \Vr_j &&\text{for }i\in\{0,1,\dots,i-1\},\\
    Y(i) &= \left\{(x,i)\in\Vr_i\,:\,2^i\ell\leq x\leq 2^i\ell+ J_{i+1}-1\text{ for some }\ell\right\}.
  \end{align*}
  Furthermore, let $Z$ be the set of vertices on level $i$ that can be forced by vertices in
  $Y(i-1)=\Vr_{i-1}$ when $Y$ is the set of colored vertices, that is,
  $Z = \left\{(x,i)\ :\ \{(x,i)\}=N(v)\setminus Y\text{ for some }v=(y,i-1)\in Y\right\}$. Then
  $Y(i)\cup Z=\Vr_i$.
\end{lemma}
\begin{proof}
  Let $(x,i)$ be an arbitrary element of $\Vr_i$. We need to show that $(x,i)\in Y(i)\cup Z$. If
  $2^{i}\ell\leqslant x\leqslant 2^{i}\ell+J_{i+1}-1$ for some $\ell$, then $(x,i)\in
  Y(i)$. Otherwise $2^{i}\ell+J_{i+1}\leqslant x\leqslant 2^{i}(\ell+1)-1$ for some $\ell$. Setting
  $y=x -2^{i-1}$, and observing that $0\leq J_{i+1}-2^{i-1}\leq y<x\leq 2^r-1$, we obtain
  $(y,i-1)\in \Vr_{i-1}=Y(i-1)$. The only neighbors of $(y,i-1)$ that are potentially not in $Y$ are
  $(x,i)$ and $(y,i)$. Using~\eqref{eq:Jacobsthal_bound}, we obtain
  \[2^i\ell\leq 2^i\ell+J_{i-1}\stackrel{\eqref{eq:Jacobsthal_bound}}{=}2^i\ell+J_{i+1}-2^{i-1}\leq\
    y\ \leq
    2^i\ell+2^{i-1}-1\stackrel{\eqref{eq:Jacobsthal_bound}}{=}2^i\ell+J_{i+1}-J_{i-1}-1\leq
    2^i\ell+J_{i+1}-1,\]
  and therefore, $(y,i)\in Y(i)$. As a consequence $\{(x,i)\}=N((y,i-1))\setminus Y$, which implies
  $(x,i)\in Z$.  
\end{proof}
\begin{lemma}\label{lem:prop_up}
  For $k\in\{0,1,\dotsc,r\}$, $X_{r+k}=X_{r+k}(0)\cup X_{r+k}(1)\cup\cdots\cup X_{r+k}(r)$ with
\[X_{r+k}(i) =
\begin{cases}
\Vr_i &\text{for }i\in\{0,\dotsc,k\},\\  
X_r(i) &\text{for }i\in\{k+1,k+2,\dotsc,r\}.
\end{cases}\]
\end{lemma}
\begin{proof}
  We proceed by induction on $k$. For $k=0$, there is nothing to do since $X_r(0)=\Vr_0$ by Lemma~
  \ref{lem:prop_down}. For $k\geqslant 1$, the sets $X_{r+k}$ and $X_{r+k-1}$ differ only on level
  $k$ (see~\eqref{eq:prop_up}), and by induction, 
  \begin{align*}
  X_{r+k}(i) &=X_{r+k-1}(i)= \Vr_i &&\text{for }i\in\{0,1,\dotsc,k-1\}\\
  X_{r+k}(i) &=X_{r+k-1}(i)= X_r(i) &&\text{for }i\in\{k+1,\dotsc,r-1\}.
  \end{align*} 
We can apply Lemma~\ref{lem:prop_up_aux} with $i=k$ and $Y=X_{r+k-1}(0)\cup X_{r+k-1}(1)\cup\dots\cup
X_{r+k-1}(k)$ to obtain
\begin{multline*}
  X_{r+k}(k)\stackrel{\eqref{eq:prop_up}}{=}X_{r+k-1}(k)\cup\left\{(x,k)\in\Vr_{k}\ :\ \{(x,k)\}=N(v)\setminus X_{r+k-1}\text{ for some
    }v=(y,k-1)\in X_{r+k-1}\right\}\\
  =\Vr_k.\qedhere
\end{multline*}
\end{proof}
Combining Lemmas~\ref{lem:prop_down} and~\ref{lem:prop_up}, we have proved that $\Sr$ is indeed a
zero forcing  set for $\BF(r)$.
\begin{lemma}\label{lem:propagating}
For every $r\geqslant 1$, $\Sr$ is a zero forcing  set for the butterfly network $\BF(r)$. \qed
\end{lemma}
From Lemmas~\ref{lem:cardinality} and~\ref{lem:propagating} we obtain an upper
bound for the zero forcing number of the butterfly network.
\begin{proposition}\label{prop:lower_bound}
For every $r\geqslant 1$, $\displaystyle Z\left(\BF(r)\right)\leqslant\frac19\left[(3r+7)2^r+2(-1)^r\right]$. \qed
\end{proposition}

\section{The lower bound for $Z(\BF(r))$}\label{sec:lower_bound}
By~\eqref{eq:Zfnullity}, the corank of the adjacency matrix of a graph $G$ provides a lower bound
for the zero forcing number of $G$, and consequently we can conclude the proof of
Theorem~\ref{thm:main_result} by establishing the following result.
\begin{proposition}\label{prop:upper_bound}
Let $F$ be a field, and let $A_r$ denote the adjacency matrix of $\BF(r)$ over $F$. Then
\[\rank(A_r)\leqslant(r+1)2^r-\frac19\left[(3r+7)2^r+2(-1)^r\right]=\frac29\left[(3r+1)2^{r}-(-1)^r\right].\]
\end{proposition}
We will prove this by verifying that the rows corresponding to vertices in $\Sr$ are linear
combinations of the rows corresponding to vertices in the complement of $\Sr$.  For this purpose, it
turns out to be convenient to number the vertices recursively as indicated in
Figure~\ref{fig:butterfly_3}. Formally, this vertex numbering is given by a bijection
$f:\{0,1,2,\dotsc\}^2\to\{1,2,3,\dotsc\}$ defined as follows. For a positive integer $x$, let
$g(x)=\lfloor\log_2(x)+1\rfloor$, i.e., $g(x)$ is the unique integer such that
$2^{g(x)-1}\leqslant x<2^{g(x)}$. In addition, let $g(0)=-1$. Then,
\begin{equation}\label{eq:vertex_numbers}
f(x,i)=
\begin{cases}
  i2^i+x+1 & \text{if }i\geqslant g(x),\\
  g(x)2^{g(x)-1}+f\left(x-2^{g(x)-1},\,i\right) & \text{if }i<g(x).
\end{cases}
\end{equation}
Note that the first argument of $f$ on the right hand side is smaller than on the left hand side,
and this implies that the function $f$ is indeed well-defined by~\eqref{eq:vertex_numbers}. For
example, using $g(6)=3$ and $g(2)=2$,
\begin{align*}
  f(6,3) &= 3\times 8+6+1 =31,\\
  f(6,2) &= 3\times 4+f(2,2)=12+2\times 4+2+1=23,\\
  f(6,1) &= 3\times 4+f(2,1)= 12+2\times 2+f(0,1)=16+1\times 2+0+1=19.
\end{align*}

With respect to the vertex numbering given by~\eqref{eq:vertex_numbers}, the adjacency matrices for
$\BF(1)$ and $\BF(2)$ are \setcounter{MaxMatrixCols}{32}
\begin{align*}
A_1 &=
  \begin{pmatrix}
    0&0&1&1\\
    0&0&1&1\\
    1&1&0&0\\
    1&1&0&0
  \end{pmatrix},
&
A_2 &= \left(
  \begin{array}{cccc|cccc|cccc}
    0&0&1&1&0&0&0&0&0&0&0&0\\
    0&0&1&1&0&0&0&0&0&0&0&0\\ \cline{9-12} &&&&&&&&&&& \\[-2.2ex]
    1&1&0&0&0&0&0&0&1&0&1&0\\
    1&1&0&0&0&0&0&0&0&1&0&1\\ \hline &&&&&&&&&&& \\[-2.2ex]
    0&0&0&0&0&0&1&1&0&0&0&0\\
    0&0&0&0&0&0&1&1&0&0&0&0\\ \cline{9-12} &&&&&&&&&&& \\[-2.2ex]
    0&0&0&0&1&1&0&0&1&0&1&0\\
    0&0&0&0&1&1&0&0&0&1&0&1\\\hline &&\multicolumn{1}{|c}{}&&&&\multicolumn{1}{|c}{}&&&&& \\[-2.2ex]
    0&0&\multicolumn{1}{|c}{1}&0&0&0&\multicolumn{1}{|c}{1}&0&0&0&0&0\\
    0&0&\multicolumn{1}{|c}{0}&1&0&0&\multicolumn{1}{|c}{0}&1&0&0&0&0\\
    0&0&\multicolumn{1}{|c}{1}&0&0&0&\multicolumn{1}{|c}{1}&0&0&0&0&0\\
    0&0&\multicolumn{1}{|c}{0}&1&0&0&\multicolumn{1}{|c}{0}&1&0&0&0&0
  \end{array}\right),  
\end{align*}
and in general, $A_r$ has the structure illustrated in Figure~\ref{matrix_structure} where $I$ is the identity matrix of size $2^{r-1}\times 2^{r-1}$. 
\begin{figure}[htb]
  \centering
\end{figure}

\begin{figure}[htb]
  \begin{minipage}[t]{.48\linewidth}
    \includegraphics[width=\textwidth]{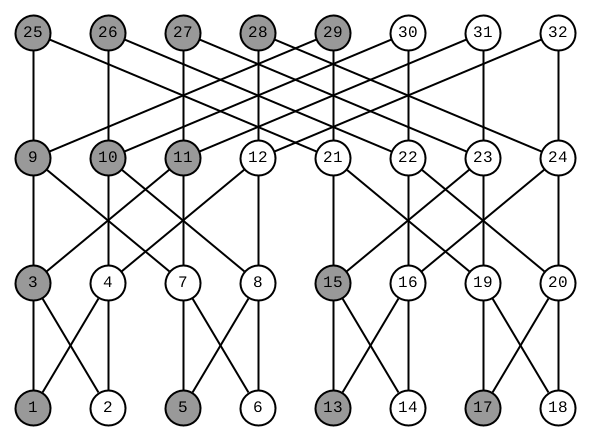}
  \caption{The butterfly network $\BF(3)$. The zero forcing set $S^{(3)}$ is indicated by darker vertices.}\label{fig:butterfly_3}
  \end{minipage}
  \begin{minipage}[t]{.5\linewidth}
\centering
\begin{tikzpicture}[scale=.6]
  \draw (0,0) -- (10,0) -- (10,10) --(0,10) -- (0,0);
  \draw (0,6) -- (4,6) -- (4,10);
  \draw (8,10) -- (8,7) -- (10,7);
  \draw (10,6) -- (8,6) -- (8,3) -- (10,3);
  \draw (4,0) -- (4,2) -- (7,2) -- (7,0);
  \draw (0,2) -- (3,2) -- (3,0);
  \draw (8,0) -- (8,2) -- (10,2);
  \draw (8,0) -- (8,10);
  \draw (0,2) -- (10,2);
  \draw (4,6) -- (8,6) -- (8,2) -- (4,2) -- (4,6);
  \node at (2,8) {{\Large $A_{r-1}$}};
  \node at (6,4) {{\Large $A_{r-1}$}};
  \node at (6,8) {{\Large $\mathbf{0}$}};
  \node at (2,4) {{\Large $\mathbf{0}$}};
  \node at (9,8.5) {{\Large $\mathbf{0}$}};
  \node at (1.5,1) {{\Large $\mathbf{0}$}};
  \node at (5.5,1) {{\Large $\mathbf{0}$}};
  \node at (9,4.5) {{\Large $\mathbf{0}$}};
  \node at (9,1) {{\Large $\mathbf{0}$}};
  \node at (3.5,1.5) {{\Large $I$}};
  \node at (3.5,.5) {{\Large $I$}};
  \node at (7.5,1.5) {{\Large $I$}};
  \node at (7.5,.5) {{\Large $I$}};
  \node at (8.5,2.5) {{\Large $I$}};
  \node at (9.5,2.5) {{\Large $I$}};
  \node at (8.5,6.5) {{\Large $I$}};
  \node at (9.5,6.5) {{\Large $I$}};
\end{tikzpicture}  
    \caption{The structure of the adjacency matrix of the butterfly network $\BF(r)$.}\label{matrix_structure}    
  \end{minipage}
\end{figure}
Using the vertex numbering given by~(\ref{eq:vertex_numbers}), the upper bound construction for a
zero forcing set $\Sr\in V(\BF(r))$ can be written recursively as $S^{(1)} = \{1,3\}$ and
\begin{equation}\label{eq:recursion}
S^{(r)} = S^{(r-1)}\cup\left\{i+r2^{r-1}\ :\ i\in S^{(r-1)},\,i\leqslant
  (r-1)2^{r-1}\right\}\cup\left\{r2^r+1,\dotsc,r2^r+J_{r+1}\right\}  
\end{equation}
for $r\geqslant 2$. In order to prove Proposition~\ref{prop:upper_bound} it is
sufficient to show that every row $i\in\Sr$ of $A_r$ can be written as a linear combination of the
rows in $\overline S^{(r)}=\{1,\dotsc,(r+1)2^r\}\setminus\Sr$. We proceed by induction on $r$. Let
$A_r(i)$ denote the $i$-th row of $A_r$. The induction base is provided by checking the cases $r=1$
and $r=2$. For $S^{(1)}=\{1,3\}$, we have
\begin{align}
  A_1(1) &= A_1(2), \label{eq:LB_base_1}\\ 
  A_1(3) &= A_1(4),\label{eq:LB_base_2}
\end{align}
and for $S^{(2)}=\{1,5,3,9,10,11\}$ (listed level by level) we have
\begin{align}
  A_2(1) &= A_2(2), \label{eq:LB_base_3}\\
  A_2(5) &= A_2(6), \label{eq:LB_base_4}\\
  A_2(3) &= A_2(4)+A_2(7)-A_2(8), \label{eq:LB_base_5}\\
  A_2(9) &= A_2(2) + A_2(6) - A_2(12), \label{eq:LB_base_6}\\
  A_2(10) &= A_2(12), \label{eq:LB_base_7}\\
  A_2(11) &= A_2(2)+A_2(6)-A_2(12). \label{eq:LB_base_8}
\end{align}
The next two lemmas follow directly from the recursive structure illustrated in Figure~\ref{matrix_structure}.
\begin{lemma}\label{lem:induction_1}
  If $i\leqslant (r-1)2^{r-1}$ and 
\[A_{r-1}(i)=\sum_{j\in K^+}A_{r-1}(j)-\sum_{j\in K^-}A_{r-1}(j)\text{ for some }
K^+,K^-\subseteq\overline S^{(r-1)},\] 
then
  \begin{align*}
A_{r}(i) &= \sum_{j\in K^+}A_{r}(j)-\sum_{j\in K^-}A_{r}(j)\quad\text{and}\\
A_{r}\left(i+r2^{r-1}\right) &= \sum_{j\in K'^+}A_{r}(j)-\sum_{j\in K'^-}A_{r}(j)   
  \end{align*}
where $K^+,K^-\subseteq\overline S^{(r)}$ and $K'^{\varepsilon}=\{j+r2^{r-1}\ :\ j\in
K^\varepsilon\}\subseteq\overline S^{(r)}$ for $\varepsilon\in\{+,-\}$. \qed
\end{lemma}
\begin{lemma}\label{lem:induction_2}
If $(r-1)2^{r-1}+1\leqslant i\leqslant (r-1)2^{r-1}+J_r$ and 
\[A_{r-1}(i)=\sum_{j\in K^+}A_{r-1}(j)-\sum_{j\in K^-}A_{r-1}(j)\text{ for some }
K^+,K^-\subseteq\overline S^{(r-1)},\] 
then
\[A_{r}(i)=\sum_{j\in K'^+}A_{r}(j)-\sum_{j\in K'^-}A_{r}(j)\]
where 
\begin{align*}
  K'^+ &= K^+\cup\left\{j+r2^{r-1}\ :\ j\in
    K^-\right\}\cup\left\{i+r2^{r-1}\right\}\subseteq\overline S^{(r)},\\
  K'^- &= K^-\cup\left\{j+r2^{r-1}\ :\ j\in K^+\right\}\subseteq\overline S^{(r)}.
\end{align*}\qed
\end{lemma}
To illustrate the step from $A_{r-1}(i)$ to $A_r(i)$ in Lemma~\ref{lem:induction_2}, consider $r=3$
and $i=10$. From~\eqref{eq:LB_base_7}, we have $K^+=\{12\}$ and $K^-=\emptyset$, and by
Lemma~\ref{lem:induction_2} (note that $r2^{r-1}=12$), $K'^+=\{12,\,22\}$ and $K'^-=\{24\}$, hence
$A_3(10)=A_3(12)+A_3(22)-A_3(24)$. This can be seen in terms of the recursive structure indicated in
Figure~\ref{matrix_structure} as follows. We have $A_2(10)=A_2(12)=\vect e_4+\vect e_8$, and then,
\begin{align*}
  A_3(10) &= \vect e_4+\vect e_8+\vect e_{26} +\vect e_{30},&
                                                              A_3(12) &= \vect e_4+\vect e_8+\vect e_{28} +\vect e_{32},\\
  A_3(22) &= \vect e_{16}+\vect e_{20}+\vect e_{26} +\vect e_{30},&
                                                                    A_3(24) &= \vect e_{16}+\vect e_{20}+\vect e_{28} +\vect e_{32}.
\end{align*}

Lemmas~\ref{lem:induction_1} and~\ref{lem:induction_2} take care of the first two components in the
recursion for $\Sr$ in~\eqref{eq:recursion}. It remains to check the rows $r2^r+i$ for
$i\in\left\{1,\dotsc,J_{r+1}\right\}$.  For $J_{r-1}+1\leqslant i\leqslant 2^{r-1}$ the required
linear dependence is $A_r\left(r2^r+i\right)=A_r\left(r2^r+i+2^{r-1}\right)$, because
$i+2^{r-1}>J_{r+1}$ and therefore $r2^r+i+2^{r-1}\in\overline S^{(r)}$. For $i>2^{r-1}$ we have
$i\leqslant 2^{r-1}+J_{r-1}$ and $A_r(r2^r+i)=A_r(r2^r+i-2^{r-1})$, and consequently it is
sufficient to consider $i\in\left\{1,\dotsc,J_{r-1}\right\}$. The induction step for these cases
will be from $\BF(r-2)$ to $\BF(r)$, so we have to take the recursion for the adjacency matrix one
step further which is illustrated in Figure~\ref{fig:recursion_2}.
\begin{figure}[htb]
  \centering
  \begin{tikzpicture}[scale=.6]
    \draw[thick] (0,0) rectangle (20,20);
    \draw[thick] (3,20) -- (3,17) -- (0,17);
    \draw[thick] (3,17) rectangle (6,14);
    \draw[thick] (8,20) -- (8,12) -- (0,12);
    \draw[thick] (8,12) rectangle (16,4);
    \draw[thick] (11,12) -- (11,9) -- (8,9);
    \draw[thick] (11,9) rectangle (14,6);
    \draw[thick] (0,4) -- (20,4);
    \draw[thick] (16,0) -- (16,20);
    \draw[thick] (0,14) -- (8,14);
    \draw[thick] (6,12) -- (6,20);
    \draw[thick] (6,18) rectangle (8,17);
    \draw[thick] (2,14) rectangle (3,12);
    \draw[thick] (6,15) -- (8,15);
    \draw[thick] (5,14) -- (5,12);
    \draw[thick] (8,6) -- (16,6);
    \draw[thick] (14,4) -- (14,12);
    \draw[thick] (14,10) rectangle (16,9);
    \draw[thick] (14,7) -- (16,7);
    \draw[thick] (13,4) -- (13,6);
    \draw[thick] (10,4) rectangle (11,6);
    \draw[thick] (16,6) -- (20,6);
    \draw[thick] (14,0) -- (14,4);
    \draw[thick] (6,4) rectangle (8,0);
    \draw[thick] (16,14) rectangle (20,12);
    \draw[->] (-1,18.5) to node[sloped,below] {$(r-1)2^{r-2}$} (-1,15.5);
    \draw[->] (-.7,18.5) to node[sloped,below] {$(3r-1)2^{r-2}$} (-.7,7.5);
    \draw[->] (-.4,17.5) to node[sloped,below,near end] {$(3r+4)2^{r-2}$} (-.4,1);
    \node at (6.5,17.5) {$I$};
    \node at (7.5,17.5) {$I$};
    \node at (6.5,14.5) {$I$};
    \node at (7.5,14.5) {$I$};
    \node at (2.5,13.5) {$I$};
    \node at (2.5,12.5) {$I$};
    \node at (5.5,13.5) {$I$};
    \node at (5.5,12.5) {$I$};
    \node at (10.5,4.5) {$I$};
    \node at (10.5,5.5) {$I$};
    \node at (13.5,4.5) {$I$};
    \node at (13.5,5.5) {$I$};
    \node at (14.5,6.5) {$I$};
    \node at (15.5,6.5) {$I$};
    \node at (14.5,9.5) {$I$};
    \node at (15.5,9.5) {$I$};
    \node at (17,13) {{\Large $I$}};
    \node at (19,13) {{\Large $I$}};
    \node at (17,5) {{\Large $I$}};
    \node at (19,5) {{\Large $I$}};
    \node at (7,3) {{\Large $I$}};
    \node at (7,1) {{\Large $I$}};
    \node at (15,3) {{\Large $I$}};
    \node at (15,1) {{\Large $I$}};
    \node at (4.5,18.5) {{\Large $0$}};
    \node at (1.5,15.5) {{\Large $0$}};
    \node at (4.5,15.5) {{\Large $A_{r-2}$}};
    \node at (1.5,18.5) {{\Large $A_{r-2}$}};
    \node at (12.5,10.5) {{\Large $0$}};
    \node at (9.5,7.5) {{\Large $0$}};
    \node at (12.5,7.5) {{\Large $A_{r-2}$}};
    \node at (9.5,10.5) {{\Large $A_{r-2}$}};
    \node at (3,2) {{\Large $0$}};
    \node at (11,2) {{\Large $0$}};
    \node at (18,2) {{\Large $0$}};
    \node at (18,9) {{\Large $0$}};
    \node at (18,17) {{\Large $0$}};
    \node at (12,16) {{\Large $0$}};
    \node at (4,8) {{\Large $0$}};
    \node at (22,18) {$(r-2)2^{r-2}$};
    \node at (22,17) {$(r-1)2^{r-2}$};
    \node at (22,15) {$(2r-3)2^{r-2}$};
    \node at (22,14) {$(r-1)2^{r-1}$};
    \node at (22,12) {$r2^{r-1}$};
    \node at (22,10) {$(3r-2)2^{r-2}$};
    \node at (22,9) {$(3r-1)2^{r-2}$};
    \node at (22,7) {$(4r-3)2^{r-2}$};
    \node at (22,6) {$(2r-1)2^{r-1}$};
    \node at (22,4) {$r2^r$};
    \node at (22,2) {$(2r+1)2^{r-1}$};
    \node at (22,0) {$(r+1)2^r$}; 
  \end{tikzpicture}
  \caption{The second level of the recursion for $A_r$.}
  \label{fig:recursion_2}
\end{figure}
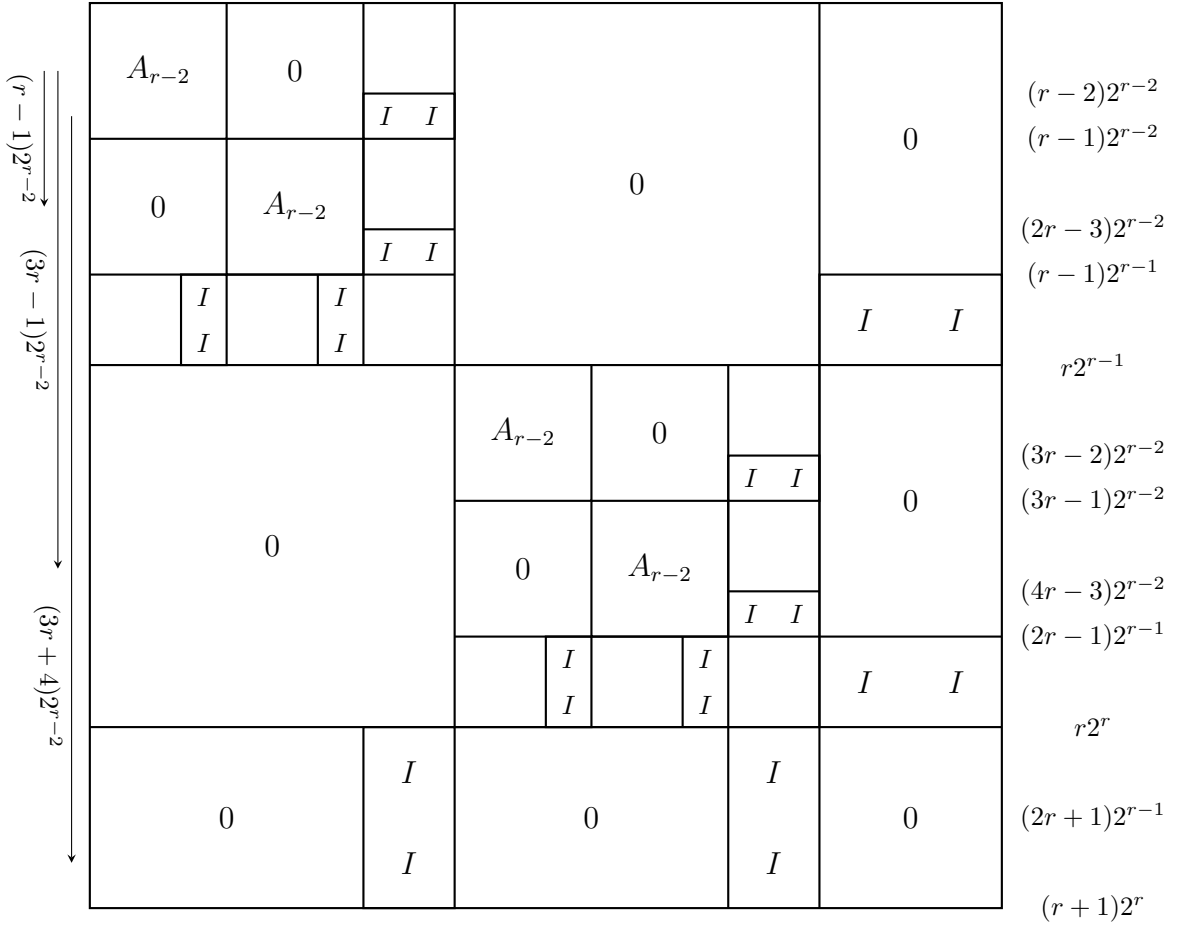
The basic idea is as follows. Let $i\in\{1,\dotsc,J_{r-1}\}$. Then $(r-2)2^{r-2}+i\in S^{(r-2)}$,
and by induction there are sets $K^+,\,K^-\subseteq\overline S^{(r-2)}$ such that
\begin{equation}\label{eq:hypo}
A_{r-2}\left((r-2)2^{r-2}+i\right)=\sum_{j\in K^+}A_{r-2}(j)-\sum_{j\in K^-}A_{r-2}(j),  
  \end{equation}
or equivalently
\begin{equation}\label{eq:hypo_a}
\sum_{j\in K^+}A_{r-2}(j)-\sum_{j\in K'^-}A_{r-2}(j)=\vect 0,  
\end{equation}
where $K'^-=K^-\cup\left\{(r-2)2^{r-2}+i\right\}$. This is a linear dependence of the rows of
$A_{r-2}$ with coefficients in $\{1,-1\}$ and involving exactly one of the rows
$(r-2)2^{r-2}+1,\dotsc,(r-2)2^{r-2}+J_{r-1}$, namely $(r-2)2^{r-2}+i$. Putting $K=K^+\cup K'^-$ we
have
\begin{equation}\label{eq:large_enough}
  K\cap\left\{(r-2)2^{r-2}+1,\dotsc,(r-2)2^{r-2}+J_{r-1}\right\}=\left\{(r-2)2^{r-2}+i\right\}.
\end{equation}
We now translate the $\lvert K\rvert$ rows in this linear dependence by $(r-1)2^{r-2}$ and
$(3r-1)2^{r-2}$ as indicated in Figure~\ref{fig:recursion_2}. The combination of the $2\lvert
K\rvert$ translated rows is a $\{0,1,-1\}$-vector $\vect x$ which has all its nonzero entries in
columns with indices in $\{(r-1)2^{r-1}+1,\dotsc,r2^{r-1}\}\cup\{(2r-1)2^{r-1}+1,\dotsc,r2^r\}$, and
has $x_k=1$ for $k\in\left\{(r-1)2^{r-1}+i,\,(2r-1)2^{r-1}+i\right\}$ which are the one-entries of
the row $A_r(r2^r+i)$. Finally, we use some of the rows $r2^r+J_{r+1}+1,\dotsc,(r+1)2^r$ with the
appropriate sign to eliminate the other nonzero entries of $\vect x$.

More precisely, we define
$\tilde K=\tilde K^+\cup\tilde K^-\subseteq\{1,\dotsc,(r+1)2^r\}$ with $\tilde K^+=\tilde K^+_1\cup
\tilde K^+_2$ and $\tilde K^-=\tilde K^-_1\cup\tilde K^-_2$  where
\begin{align}
\tilde K^+_1&= \left\{j+(r-1)2^{r-2}\ :\ j\in K'^-\right\}\cup\left\{j+(3r-1)2^{r-2}\ :\ j\in K'^-\right\}\label{eq:K1plus}\\
\tilde K^-_1&= \left\{j+(r-1)2^{r-2}\ :\ j\in K^+\right\}\cup\left\{j+(3r-1)2^{r-2}\ :\ j\in
  K^+\right\}\label{eq:K1minus}\\
\nonumber \tilde K^+_2&=\left\{j+(3r+4)2^{r-2}\ :\ j\in K^+\text{ with } j>(r-2)2^{r-2}\right\}\\
&\qquad\cup\left\{j+(3r+5)2^{r-2}\ :\ j\in K^+\text{ with } j>(r-2)2^{r-2}\right\}.\label{eq:K2plus}\\  
\nonumber \tilde K^-_2&=\left\{j+(3r+4)2^{r-2}\ :\ j\in K^-\text{ with } j>(r-2)2^{r-2}\right\}\\
&\qquad\cup\left\{j+(3r+5)2^{r-2}\ :\ j\in K'^-\text{ with } j>(r-2)2^{r-2}\right\}.\label{eq:K2minus}  
\end{align}
The construction of $\tilde K$ is illustrated for $r=4$ and $i=1$ in Figure~\ref{fig:induction}.
\begin{figure}[htb]
  \centering
  \includegraphics[width=\textwidth]{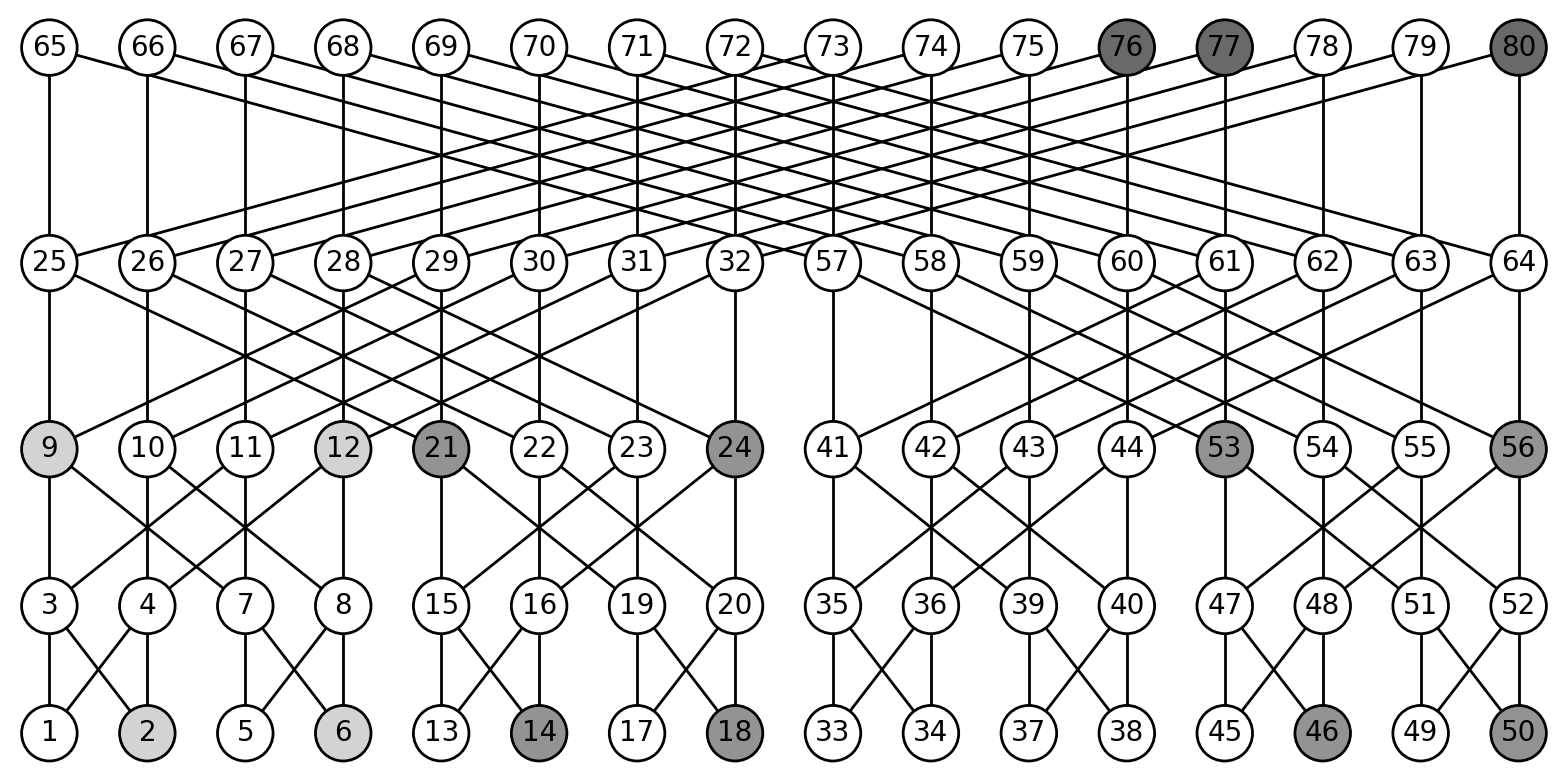}
  \caption{The construction of $\tilde K$ for $r=4$ and $i=1$. Here $K^+=\{2,6\}$, $K'^-=\{9,12\}$,
    $\tilde K^-_1=\{14,18,46,50\}$, $\tilde K^+_1=\{21,24,53,56\}$, $\tilde K^-_2=\{76,77,80\}$,
    $\tilde K^+_2=\emptyset$.}
  \label{fig:induction}
\end{figure}
Here, we want to construct a linear combination representing $A_4(65)$, and we start
with~\eqref{eq:LB_base_6} in the form $A_2(2)+A_2(6)-A_2(9)-A_2(12)=\vect 0$ with $K^+=\{2,6\}$ and $K'^-=\{9,12\}$. Shifting by $12$
and $44$, we obtain $\tilde K^-=\{14,18,46,50\}$ and $\tilde K^+=\{21,24,53,56\}$, and looking
at Figure~\ref{fig:recursion_2}, we can use $A_2(2)+A_2(6)-A_2(9)-A_2(12)=\vect 0$ to verify that
\begin{multline*}
  A_4(21)+A_4(24)+A_4(53)+A_4(56)-A_4(14)-A_4(18)-A_4(46)-A_4(50)\\
  =\vect e_{25} +\vect e_{28}+\vect e_{29}+\vect e_{32}+\vect e_{57} +\vect e_{60}+\vect e_{61}+\vect e_{64},
\end{multline*}
where the right hand side corresponds to the (level 3)-neighborhood of $\{21,24,53,56\}$ which is the
set of indices coming from shifting $9$ and $12$. The contributions of rows from shifting $2$ and
$6$ cancel, because these rows do not see the identity matrices in columns 25 to 32 (see
Figure~\ref{matrix_structure}), which reflects the fact that vertices 2 and 6 don't have neighbors
on level $3$. Now we note that the neighborhood of $\{76,77,80\}$ is $\{28,29,32,60,61,64\}$, and
conclude
\begin{multline*}
  A_4(21)+A_4(24)+A_4(53)+A_4(56)-A_4(14)-A_4(18)-A_4(46)-A_4(50)-A_4(76)-A_4(77)-A_4(80)\\
  =\vect e_{25} +\vect e_{57} =A_4(65),
\end{multline*}
as required.

\medskip

Lemmas~\ref{lem:induction_3} and~\ref{lem:induction_4} state that $\tilde K$ has the required properties.
\begin{lemma}\label{lem:induction_3}
Let $r\geqslant 3$, $i\in\left\{1,\dotsc,J_{r-1}\right\}$, suppose $K\subseteq\overline S^{(r-2)}$
satisfies~\eqref{eq:hypo}, and define $\tilde K$ by~\eqref{eq:K1plus} to~\eqref{eq:K2minus}. Then
$\tilde K\subseteq\overline S^{(r)}$. 
\end{lemma}
\begin{proof}
Note that by construction
\begin{equation}\label{eq:box}
  \tilde K_1=\tilde K^+_1\cup \tilde K^-_1\subseteq\left\{(r-1)2^{r-2}+1,\dots,(r-1)2^{r-1}\right\}\cup\left\{(3r-1)2^{r-2}+1,\dots,(2r-1)2^{r-1}\right\}.
\end{equation}
Suppose there is an element $j\in K$ such that $k=j+(r-1)2^{r-2}\in\tilde K_1\cap\Sr$. From $j\in
K\subseteq V^{(r-2)}$ it follows that $k\leq (r-1)2^{r-2}+(r-1)2^{r-2}=(r-1)2^{r-1}<r2^{r-1}$.
Using~\eqref{eq:recursion},
we obtain 
\begin{align*}
k&\in\Sr=S^{(r-1)}\cup\left\{p+r2^{r-1}\ :\ p\in S^{(r-1)},\,p\leqslant
  (r-1)2^{r-1}\right\}\cup\left\{r2^r+1,\dotsc,r2^r+J_{r+1}\right\}\\
\stackrel{k\leq r2^{r-1}}{\implies} k&\in S^{(r-1)}=S^{(r-2)}\cup\left\{p+(r-1)2^{r-2}\ :\ p\in S^{(r-2)},\,p\leqslant
  (r-2)2^{r-2}\right\}\\
&\qquad\qquad\qquad\qquad\cup\left\{(r-1)2^{r-2}+1,\dotsc,(r-1)2^{r-1}+J_{r}\right\}\\
\implies k&= p+(r-1)2^{r-2}\text{ for some } p\in S^{(r-2)}\text{ with }p\leqslant
  (r-2)2^{r-2},  
\end{align*}
which contradicts the assumption that $j\in K\subseteq\overline S^{(r-2)}\cup\{(r-2)2^{r-2}+i\}$.
Similarly, for $k=j+(3r-1)2^{r-1}\in \tilde K_1\cap\Sr$ we obtain
\begin{align*}
k&\in\Sr=S^{(r-1)}\cup\left\{p+r2^{r-1}\ :\ p\in S^{(r-1)},\,p\leqslant
  (r-1)2^{r-1}\right\}\cup\left\{r2^r+1,\dotsc,r2^r+J_{r+1}\right\}\\
\implies k &=p+r2^{r-1}\text{ for some }p\in S^{(r-1)}\text{ with }p\leqslant
  (r-1)2^{r-1}\\
\implies k&= q+(r-1)2^{r-2}\text{ for some } q\in S^{(r-2)}\text{ with }q\leqslant
  (r-2)2^{r-2},
\end{align*}
where we use $k>(3r-1)2^{r-2}$ for the last implication. Again we obtain a contradiction to the
assumption that $j\in K\subseteq\overline S^{(r-2)}\cup\{(r-2)2^{r-2}+i\}$. Finally, the elements
of $\tilde K_2=\tilde K^+_2\cup\tilde K^-_2$ are in $\overline S^{(r)}$ since for $j\in K^+\cup K^-$ we have
\[j>(r-2)2^{r-2}\implies j>(r-2)2^{r-2}+J_{r-1}\implies
j+(3r+4)2^{r-2}>r2^r+2^{r-1}+J_{r-1}=2^r+J_{r+1},\]
and for $j\in K'$,
\[j>(r-2)2^{r-2}\implies j+(3r+5)2^{r-2}>r2^r+2^{r-1}+2^{r-2}>r2^r+J_{r+1},\]
and this concludes the proof of the lemma.
\end{proof}
\begin{lemma}\label{lem:induction_4}
Let $r\geqslant 3$, $i\in\left\{1,\dotsc,J_{r-1}\right\}$, suppose $K^+,K^-\subseteq\overline S^{(r-2)}$
satisfy~\eqref{eq:hypo}, and define $\tilde K^+$ and $\tilde K^-$ by~\eqref{eq:K1plus} to~\eqref{eq:K2minus}.
Then 
\begin{equation}\label{eq:punchline}
A_{r}\left(r2^r+i\right)=\sum_{j\in \tilde K^+}A_{r}(j)-\sum_{j\in \tilde K^-}A_{r}(j).
\end{equation}
\end{lemma}
\begin{proof}
Setting
\begin{align*}
\vect x &= \sum_{j\in \tilde K^+_1}A_r(j)-\sum_{j\in\tilde K^-_1}A_r(j),& 
\vect y &= A_r(r2^r+i)-\sum_{j\in \tilde K^+_2}A_r(j)+\sum_{j\in \tilde K^-_2}A_r(j) 
\end{align*}
equation~\eqref{eq:punchline} is equivalent to $\vect x=\vect y$. From~\eqref{eq:box} and~\eqref{eq:hypo} it follows that
\[\supp(\vect x)\subseteq\left\{(r-1)2^{r-1}+1,\dots,r2^{r-1}\right\}\cup\left\{(2r-1)2^{r-1}+1,\dots,r2^{r}\right\},\]
and by construction, for every $j\in\{1,\dotsc,2^{r-2}\}$, 
\[\vect x\left((r-1)2^{r-1}+j\right)=\vect x\left((r-1)2^{r-1}+2^{r-2}+j\right)=\vect x\left((2r-1)2^{r-1}+j\right)=x\left((2r-1)2^{r-1}+2^{r-2}+j\right).\]
Denoting this value by $\tilde x(j)$, we have
\begin{equation}\label{eq:char_x}
\tilde{\vect x}(j) =
\begin{cases}
  1 & \text{if }(r-2)2^{r-2}+j\in K'^-,\\
  -1 & \text{if }(r-2)2^{r-2}+j\in K^+,\\
  0 & \text{otherwise}.
\end{cases}
\end{equation}
From~\eqref{eq:K2plus} and~\eqref{eq:K2minus} it follows that
\[\tilde K^+_2\cup\tilde
K^-_2\cup\left\{(2r+1)2^{r-1}+i\right\}\subseteq\left\{(2r+1)2^{r+1}+1,\dots,(r+1)2^r\right\},\]
and therefore
\[\supp(\vect
y)\subseteq\left\{(r-1)2^{r-1}+1,\dots,r2^{r-1}\right\}\cup\left\{(2r-1)2^{r-1}+1,\dots,r2^{r}\right\},\]
After replacing $A_r\left(r2^r+i\right)$ by $A_r\left(r2^r+2^{r-1}+i\right)$ (which we can do since
the two rows are equal), the rows contributing to $\vect y$ come in pairs $\left(j,j+2^{r-2}\right)$
where both rows in each pair have the same sign in $\vect y$. Therefore,  
\[\vect y\left((r-1)2^{r-1}+j\right)=y\left((r-1)2^{r-1}+2^{r-2}+j\right)=\vect
  y\left((2r-1)2^{r-1}+j\right)=\vect y\left((2r-1)2^{r-1}+2^{r-2}+j\right).\]
For $j\in\{1,\dotsc,2^{r-2}\}$ we have $\vect y\left((r-1)2^{r-1}+j\right)=1$ if and only if
$(2r+1)2^{r-1}+j=j'+(3r+4)2^{r-1}$ for some $j'\in K'^-$, or equivalently $j'=(r-2)2^{r-2}+j\in
K'^-$. Similarly, we have $\vect y\left((r-1)2^{r-1}+j\right)=-1$ if and only if $j'=(r-2)2^{r-2}+j\in
K^+$, and comparing this with~\eqref{eq:char_x} we conclude $\vect x=\vect y$, as required.
\end{proof}
\begin{proof}[Proof of Proposition~\ref{prop:lower_bound}]
The statement follows by induction with
base~\eqref{eq:LB_base_1}--\eqref{eq:LB_base_8}, using
Lemmas~\ref{lem:induction_1},~\ref{lem:induction_2},~\ref{lem:induction_3} and~\ref{lem:induction_4} for the induction
step.  
\end{proof}
Finally, Theorem~\ref{thm:main_result} is a consequence of Propositions~\ref{prop:lower_bound} and~\ref{prop:upper_bound}.

\section{Additional comments}\label{sec:comments}
For a graph $G=(V,E)$ and a zero-forcing set $S\subseteq V$, the \emph{propagation time} $\pt(S)$
has been defined in~\cite{HogbenHuynhKingsleyMeyerWalkerYoung-2012-Propagationtimezero} as the
length $m$ of the increasing sequence $S=S_0\subsetneq S_1\subseteq\cdots\subsetneq S_m=V$, where
\[S_i=S_{i-1}\cup\left\{w\ :\ \{w\}=N(v)\cap S_{i-1}\text{ for some }v\in
    S_{i-1}\right\}\quad\text{for }i=1,2\dotsc\,.\] The propagation time of $\pt(G)$ of the graph
$G$ is the minimum of the propagation times $\pt(S)$ over all minimum zero-forcing sets $S$. The
construction in Section~\ref{sec:upper_bound} gives the upper bound $\pt(\BF(r))\leqslant 2r$, and
we leave it as an open problem to determine the propagation time of $\BF(r)$.  As mentioned in the
introduction, a concept closely related to zero-forcing, called \emph{power domination}, was
introduced in~\cite{Haynes2002a}. A vertex set $S\subseteq V$ is called power dominating if the
closed neighborhood $N[S]=S\cup\{w\,:\,vw\in E\text{ for some }v\in S\}$ is a zero forcing set. It
was shown in~\cite{Benson.etal_2015_Powerdominationzero} that $Z(G)/\Delta$ provides a lower bound
for the size of a power dominating set in $G$ where $\Delta$ is the maximum degree of $G$. This
implies that the power domination number of the butterfly network $\BF(r)$, that is, the minimum
size of a power dominating set, is at least
\[\left\lceil\frac{1}{36}\left[(3r+7)2^r+2(-1)^r\right]\right\rceil.\]
This bound does not appear to be tight and we leave for future work the problems of finding the
power domination number of the butterfly network as well as its \emph{power propagation time} which is defined
in~\cite{Ferrero2016} as
\[\ppt(G)=1+\min\{\pt\left(N[S]\right)\,:\,S\text{ is a minimum power dominating set in $G$}\}.\]

\subsection*{Acknowledgement}
We would like to thank an anonymous reviewer for carefully reading a previous version of the paper
and providing a large number of insightful comments which were incredibly helpful in clarifying the
presentation of our arguments.

\providecommand{\bysame}{\leavevmode\hbox to3em{\hrulefill}\thinspace}
\providecommand{\MR}{\relax\ifhmode\unskip\space\fi MR }
\providecommand{\MRhref}[2]{%
  \href{http://www.ams.org/mathscinet-getitem?mr=#1}{#2}
}
\providecommand{\href}[2]{#2}

\end{document}